\documentclass[11pt,twoside]{amsart}
\usepackage{t1enc}
\usepackage[latin2]{inputenc}

\usepackage{amsmath,amsfonts}
\usepackage{amsthm}
\usepackage{amscd}
\usepackage{amssymb}
\usepackage{enumerate}
\usepackage{mathrsfs}
\usepackage{dsfont}

\usepackage{stmaryrd}
\usepackage[T1]{fontenc}

\usepackage{mathtools}

\usepackage{bbm}

\usepackage{tikz}
\usepackage{tikz-cd}
\usepackage{wasysym}
\usepackage{verbatim}
\usepackage{fancyhdr}
\usepackage{fancybox}
\usepackage{url}
\usepackage{color}

\newtheorem{thm}{Theorem}[section]

\newtheorem{prop}[thm]{Proposition}
\newtheorem{cor}[thm]{Corollary}

\theoremstyle{definition}
\newtheorem{que}[thm]{Question}

\newtheorem{df}[thm]{Definition}
\newtheorem{exa}[thm]{Example}

\newtheorem{rem}[thm]{Remark}

\title{On measures induced by forcing names for ultrafilters}

\author{Piotr Borodulin--Nadzieja}
\address[Piotr Borodulin-Nadzieja]{Instytut Matematyczny, Uniwersytet Wroc\l awski, \ \ \  pl. Grunwaldzki 2/4, 50-384 Wroc\l aw, Poland}
\email{pborod@math.uni.wroc.pl}

\author{Katarzyna Cegie\l ka}
\address[Katarzyna Cegie\l ka]{Katedra Matematyki i Cybernetyki, Uniwersytet Ekonomiczny we Wroc\l awiu, ul. Komandorska 118/120, 53-345 Wroc\l aw, Poland}
\email{katarzyna.cegielka@ue.wroc.pl}

\thanks{The first author was supported by 
 National Science Center project no. 2018/29/B/ST1/00223. .}


\begin{document}

\maketitle

{\centering\footnotesize Dedicated to the memory of Kenneth Kunen (1943 -- 2020). \par}

\begin{abstract} 
	We study the interplay between properties of measures on a Boolean algebra $\mathbb{A}$ and forcing names for ultrafilters on $\mathbb{A}$. We show that several well known measure theoretic properties of Boolean algebras (such as supporting a
	strictly positive measure or carrying only separable measures) have quite natural characterizations in the forcing language.  We show some applications of this approach. In particular, we reprove a theorem of Kunen saying that in the
	classical random model there are no towers of height $\omega_2$.
\end{abstract}

In this article we try to harness measure theory to understand better models obtained by adding random reals. 


One of the main inconveniences in forcing \emph{praxis} lies in dealing with names for objects in the extension. Formally, names are quite complicated objects defined by recursion. Fortunately, in many situations, one can find a way to handle this
difficulty. E.g. for many forcing notions we may think about names for reals as Borel functions $f\colon 2^\omega \to 2^\omega$ in the ground model. In this article we exploit a way to treat names for ultrafilters. If $\mathbb{A}$ is a Boolean algebra from the ground model and
$\dot{\mathcal{U}}$ is a $\mathbb{P}$-name for an ultrafilter on $\mathbb{A}$, then we may think about $\dot{\mathcal{U}}$ as a Boolean homomorphisms $\varphi\colon \mathbb{A} \to \mathbb{P}$ (see the beginning of Section \ref{homomorphisms}). 

This approach turns
out to be quite fruitful in case when $\mathbb{P} = \mathbb{M}_\kappa$, i.e. the measure algebra of type $\kappa$  (see Section \ref{preliminaries} for the precise definitions). In the article \cite{PbnSobota} due to the first author and Damian
Sobota it was used to study sequences of ultrafilters in
the forcing extensions by measure algebras.

Here, we will rather focus on properties of single names. Every Boolean homomorphism $\varphi\colon \mathbb{A} \to \mathbb{M}_\kappa$ naturally induces a measure $\mu$ on $\mathbb{A}$, the 'pre-image' of the standard Haar measure on $\mathbb{M}_\kappa$. The question is how much information about a name for an ultrafilter we get from the measure induced by the appropriate homomorphism. 
We discuss this and related questions in Section \ref{homomorphisms}. We present several examples of names for ultrafilters on various Boolean algebras showing how the properties of the induced measures affect the properties of the ultrafilters in
the extension. We show that under $\mathsf{GCH}$ in every family of size of $\omega_2$ of homomorphisms from the Cantor algebra to $\mathbb{M}_{\omega_2}$ one can 'swap' two of them by an automorphism of $\mathbb{M}_{\omega_2}$ (Theorem \ref{Compatible}). In a sense this means
that sets of real of size $\mathfrak{c}$ in the classical
random model have to be in a sense homogeneous . Using this result we reprove the famous Kunen's theorem saying that there are no
$\omega_2$ chains in $\mathcal{P}(\omega)/Fin$ in the classical random model (Theorem \ref{towers}). Also, we show that new ultrafilters on $\mathcal{P}(\omega)/Fin \cap V$, where $V$ denotes the ground model, cannot be $\sigma$-closed (Proposition \ref{ppoint}).

Section \ref{mrp} is devoted to so called Measure Recognition Problems and its relations to random forcing. In \cite{MRP} D\v{z}amonja coined a notion of Measure Recognition Problems to describe several problems considered in measure theory.
Generally speaking, Measure Recognition Problem for a property $\varphi$ ($\mathsf{MRP}$($\varphi$)) is a question about a combinatorial characterization of Boolean algebras carrying measures having
property $\varphi$\footnote{In \cite{MRP} D\v{z}amonja assumed additionally that the measure in question should be strictly positive. Here, we drop this assumption.}. We will briefly overview the most important examples.

\begin{itemize}
	\item \textbf{strict positivity.} In \cite{Kelley} Kelley fully characterized Boolean algebras supporting strictly positive measures in a combinatorial way (using so called Kelly intersection numbers).
	
	\item \textbf{$\sigma$-additivity \& strict positivity.} The question how to characterize Boolean algebras with strictly positive $\sigma$-additive measures has a long history and it was a driving force for set theoretic measure theory for many
		years. It started with the works of von Neumann and Maharam (\cite{Maharam47}) and ended in 2006 with an article by Talagrand (\cite{Talagrand-Maharam}) with a solution to so called Control Measure Problem. See \cite{MRP} for the more
		detailed history.

	\item \textbf{separability \& strict positivity.} The question when a Boolean algebra $\mathbb{A}$ supports a separable measure was considered by M\"{a}gerl and Namioka (\cite{Namioka}) and Talagrand (\cite{Talagrand}). It is known that this
	property is very close to so called approximability, a purely combinatorial notion which itself is equivalent to weak$^*$ separability of the space of measures on $\mathbb{A}$. However, there is a $\mathsf{ZFC}$ example of an approximable Boolean
algebra which does not support a separable measure (see \cite{MirnaGrzes}).

	\item \textbf{non-separability.}  In \cite{MirnaKunen} D\v{z}amonja and Kunen studied compact spaces carrying only separable measures calling it \emph{measure separable} spaces. We will adapt this terminology for Boolean algebras. Note that if a
		Boolean algebra contains an uncountable independent family, then it is not measure separable. There are many consistent examples (see e.g. \cite{Kunen}) of Boolean algebras without uncountable independent families which do carry a
		nonseparable measure. The question if consistently every Boolean algebra without uncountable independent family is measure separable was known as the Haydon Problem. It was answered in positive by Fremlin (\cite{Fremlin}): such a
		situation holds under $\mathsf{MA}_{\omega_1}$.
	\item \textbf{non-atomicity.} By theorem of Pe\l czy\'nski and Semadeni a Boolean algebra carries an atomless measure if and only if it is not superatomic (equivalently, if its Stone space is not scattered), see \cite{Semadeni}.
\end{itemize}
In Section \ref{mrp} we obtain characterizations of some of the properties mentioned above in forcing terms. We characterize Boolean algebras with strictly positive measure by the property of being a $\sigma$-centered Boolean algebra in some random extension (Theorem
\ref{Kamburelis}), reproving a result of Kamburelis from \cite{Kamburelis}. Using the homomorphism form of names we were able to
make use of a result from ergodic theory and to simplify the original proof of Kamburelis. Also, we deal with $\mathsf{MRP}$(non-atomicity) showing that Boolean algebras carrying an atomless measures are exactly those which do not have new ultrafilters in the
random extension (Proposition \ref{super}). We show that a Boolean algebra carries a non-separable measure if and only if there is a $\mathbb{M}_{\omega_1}$-name for a ultrafilter which cannot be added by the forcing adding a single
random real (Theorem \ref{measure-separable}). We recall certain Kunen's construction to show that consistently there is a 'small' Boolean algebra $\mathbb{A}$ such that forcing with a single random real cannot add 'too generic' ultrafilters on $\mathbb{A}$
(Corollary \ref{pokunen}).

In Section \ref{nontri} we focus on sequences of homomorphisms into measure algebras ignoring its forcing interpretation. We show that on every infinite Boolean algebra one can define a sequence of homomorphisms into the measure algebra which is in a
sense non-trivially convergent: it converges pointwise but not uniformly (Theorem \ref{nontrivii}).

\section{Preliminaries}\label{preliminaries}

We consider non-zero finitely additive measures on Boolean algebras. The following theorem says that each such measure can be uniquely extended to a $\sigma$-additive Radon measure on the Stone space of $\mathbb{A}$.

\begin{thm}(see e.g. \cite[Chapter 5, Theorem 3]{Lacey})  Let $\mathbb{A}$ be a Boolean algebra and let $K$ be its Stone space. Every (finitely additive) measure $\mu$ on $\mathbb{A}$ can be uniquely extended (as a measure on $\mathrm{Clop}(K)$) to a Radon ($\sigma$-additive) measure $\hat{\mu}$
	defined on $K$. Moreover, for each $\varepsilon>0$ and a Borel set $B\subseteq K$, there is a clopen $A\subseteq K$ such that $\hat{\mu}(B \triangle A) <\varepsilon$.
\end{thm}

In what follows, for a measure $\mu$ on a Boolean algebra we will denote by $\hat{\mu}$ the extension promised by the above theorem. 

For a Boolean algebra $\mathbb{A}$ let $\mathbb{A}^+ = \mathbb{A} \setminus \{0\}$. We say that a Boolean algebra $\mathbb{A}$ \emph{supports a measure} $\mu$ if $\mu(A)>0$ for each $A\in \mathbb{A}^+$ (in this case we say that $\mu$ is
\emph{strictly positive} on $\mathbb{A}$). If $\mu$ is a measure on $\mathbb{A}$, which is not necessarily strictly positive, then we say that $\mathbb{A}$ \emph{carries} $\mu$.

	For a cardinal $\kappa$ denote by $\lambda_\kappa$ the standard (product) measure on $\{0,1\}^\kappa$. Let $\mathbb{M}_\kappa = \mathrm{Bor}(\{0,1\}^\kappa)/_{\lambda_\kappa = 0}$. We will call $\mathbb{M}_\kappa$ \emph{measure algebras}. Note
	that $\mathbb{M}_1$ is the trivial algebra and $\mathbb{M}_\omega$ is usually called \emph{the measure algebra} (and will be denoted by $\mathbb{M}$). 
	By saying that we add $\kappa$ ($\kappa>\omega$) random reals, we mean that we force with $\mathbb{M}_\kappa$. Adding a single random real means forcing with $\mathbb{M}$. 

  A measure $\mu$ on $\mathbb{A}$ is \emph{atomless} if its
  extension to $\hat{\mu}$ vanishes on points. Equivalently, $\mu$ is atomless if for each $\varepsilon>0$ there is a partition of unity $(A_n)_{n<N}$ such that $\mu(A_n)<\varepsilon$ for every $n<N$.

	If $\mathbb{A}$ carries a measure $\mu$, then the function $d_\mu\colon \mathbb{A} \to \mathbb{R}$ defined by $d_\mu(A,B) = \mu(A\triangle B)$ is a pseudo-metric (which is a metric, if $\mu$ is strictly positive). We will call it \emph{a
	Frechet-Nikodym (pseudo-)metric}. By the \emph{Maharam type of $\mu$} we understand the density of the pseudometric space $(\mathbb{A},d_\mu)$. We say that a measure $\mu$ on a Boolean algebra $\mathbb{A}$ is \emph{homogeneuous} if $\mu_{|A}$ has
	the same Maharam type for each $A\in \mathbb{A}^+$.

		We say that a homomorphism $\varphi\colon (\mathbb{A},\mu) \to (\mathbb{B},\nu)$ is \emph{metric} if $\nu(\varphi(A))=\mu(A)$ for each $A\in \mathbb{A}$. We will say that a Boolean algebra $\mathbb{A}$ is metrically isomorphic to $\mathbb{B}$ if
	it is clear which measures on $\mathbb{A}$ and $\mathbb{B}$ we mean.

\begin{thm}[Maharam's theorem] \label{Maharam} (see \cite{Maharam})
	If a Boolean algebra $\mathbb{A}$ supports a countably additive measure $\mu$, then $(\mathbb{A},\mu)$ is metrically isomorphic to a direct sum of measure algebras (with its standard measures). If $\mu$ is homogeneous of Maharam type $\kappa$, then $(\mathbb{A},\mu)$ is metrically
isomorphic to $\mathbb{M}_\kappa$.
\end{thm}

A measure $\mu$ on a Boolean algebra is \emph{separable} if $(\mathbb{A}, d_\mu)$ is separable. Note that $\mu$ on $\mathbb{A}$ is separable if and only if the space $L_1(\hat{\mu})$ is separable. It will be convenient to make the following consequence of Maharam's theorem available.

\begin{prop}\label{Maharam-conclusion} Assume that $\mu$ is a stricly positive separable measure on $\mathbb{A}$. Then there is an injective metric homomorphism $\varphi\colon (\mathbb{A}, \mu) \to (\mathbb{M},\lambda)$.
\end{prop}
\begin{proof}
	Let $K$ be the Stone space of $\mathbb{A}$. Notice that if $\mu$ is separable, then $\hat{\mu}$ is separable on $\mathrm{Bor}(K)_{/\hat{\mu} = 0}$ (since Borel sets can be approximated with respect to $\hat{\mu}$ with clopens). Then, by Theorem \ref{Maharam}, there is a metric isomorphism  $\varphi'\colon \mathrm{Bor}(K)_{/\hat{\mu} = 0} \to
	\mathbb{M}$. Let $\varphi = \varphi'_{| \mathrm{Clop}(K)}$. 
\end{proof}

A Boolean algebra $\mathbb{A}$ is \emph{$\sigma$-centered} if $\mathbb{A}^+$ is a union of countably many ultrafilters on $\mathbb{A}$, or, equivalently, if $\mathrm{St}(\mathbb{A})$ is separable. Note that if $\mathbb{A}$ is $\sigma$-centered
	Boolean algebra witnessed by $\{\mathcal{U}_n\colon n\in \omega\}$, then $\mu = \sum \delta_{\mathcal{U}_n}/2^{n+1}$ is strictly positive on $\mathbb{A}$ (and so, $\mathbb{A}$ supports a measure).

In general we denote elements of Boolean algebras, including $\mathbb{M}_\kappa$'s, with capital letters. However, speaking about an element of $\mathbb{M}_\kappa$ sometimes we understand it as a condition of the forcing notion $\mathbb{M}^+_\kappa$. In this case, we will rather use the standard forcing notation: $p$, $q$ and so on. 

By \emph{a real} we will understand an element of the Cantor set $2^\omega$ (or, equivalently, by the standard identification of $2^\omega$ with $\mathcal{P}(\omega)$, a subset of $\omega$). 

For unexplained terminology consult \cite{Kunen-ST} and \cite{Bartoszynski} (forcing),  \cite{Halmos} (measure theory).

\section{Homomorphisms and names for ultrafilters}\label{homomorphisms}

Suppose that we force by $\mathbb{M}_\kappa$ over a model $V$. Let $\mathbb{A}$ be a Boolean algebra in $V$. Then, in
$V[G]$, $\mathbb{A}$ also forms a Boolean algebra. (Formally, we should speak here about the canonical name $\check{\mathbb{A}}$ for $\mathbb{A}$ but we are going to abuse the notation by identifying objects from the ground model with its canonical
names.)  Typically, in $V[G]$ there are new ultrafilters on $\mathbb{A}$. 

A priori the access to those 'new' ultrafilters is quite remote. However, the point is that there is a natural correspondence between $\mathbb{M}_\kappa$-names of ultrafilters on $\mathbb{A}$ and homomorphisms $\varphi\colon \mathbb{A} \to \mathbb{M}_\kappa$.
Indeed, assume $\dot{\mathcal{U}}$ is a $\mathbb{M}_\kappa$-name for an ultrafilter on a Boolean algebra $\mathbb{A}$. Define $\varphi\colon \mathbb{A} \to \mathbb{M}_\kappa$ by $\varphi(A) = \llbracket A\in \dot{\mathcal{U}}\rrbracket$, where
$\llbracket \alpha \rrbracket$ is the truth value of $\alpha$ (see e.g. \cite[Definition 7.14]{Kunen}). It is plain to check
that it is a Boolean homomorphism. On the other hand, fix a homomorphism $\varphi\colon \mathbb{A} \to \mathbb{M}_\kappa$ and let \[ \dot{\mathcal{U}} = \{\langle A, \varphi(A) \rangle\colon A\in \mathbb{A}\}. \]
Then $\dot{\mathcal{U}}$ is an $\mathbb{M}_\kappa$-name for an ultrafilter on $\mathbb{A}$.

So, every homomorphism $\varphi\colon \mathbb{A} \to \mathbb{M}_\kappa$ can be interpreted as an $\mathbb{M}_\kappa$-name for an ultrafilter on $\mathbb{A}$ (and will be denoted by $\dot{\varphi}$ in this case).
Also, if $\dot{\mathcal{U}}$ is an $\mathbb{M}_\kappa$-name for an ultrafilter on $\mathbb{A}$, then we may assume that it is of the form $\dot{\varphi}$ for some homomorphism $\varphi$ (since $\Vdash_{\mathbb{M}_\kappa}
\dot{\varphi} = \dot{\mathcal{U}}$, where $\varphi$ is defined by $\varphi(A) = \llbracket A\in \dot{U}\rrbracket$).

If $\varphi\colon \mathbb{A} \to \mathbb{M}_\kappa$ is a Boolean homomorphism, then $\mu = \lambda \circ \varphi$ is a measure. We will call it \emph{the measure induced} by $\varphi$. In this way, each $\mathbb{M}_\kappa$-name $\dot{\varphi}$ for
ultrafilter on $\mathbb{A}$ induces a measure on $\mathbb{A}$.

\begin{exa}\label{old} If the range of a homomorphism $\varphi\colon \mathbb{A} \to \mathbb{M}$ is trivial, then $\mathcal{U} = \varphi^{-1}[\{1\}]$ is an ultrafilter on $\mathbb{A}$ (in $V$). In this case $\dot{\varphi}$ is a canonical name for $\mathcal{U}$
	and $\lambda\circ \varphi = \delta_\mathcal{U}$, the Dirac delta at $\mathcal{U}$.
	See also Proposition \ref{purely-atomic}.
\end{exa}

The next simple example shows that there may be two names for ultrafilters which induce the same measure but which are interpreted as distinct points by each generic filter.

\begin{exa} Consider a Boolean algebra $\mathbb{A}$ and  $\mathcal{U}_0 \ne \mathcal{U}_1 \in \mathrm{St}(\mathbb{A})$. Let $M\in \mathbb{M}$ be such that $\lambda(M)=1/2$ and for $i\in \{0,1\}$ define $\varphi_i \colon \mathbb{A} \to \mathbb{M}$ 
$$
\varphi_i(A) =
\begin{cases}
	1  \mbox{ if } A \in \mathcal{U}_1 \cap \mathcal{U}_0, \\
	M  \mbox{ if } A \in \mathcal{U}_i \setminus \mathcal{U}_{1-i},\\
	M^c\mbox{ if } A \in \mathcal{U}_{1-i} \setminus \mathcal{U}_i, \\
	0  \mbox{ otherwise. }
\end{cases}
$$
	Then $\varphi_0$ and $\varphi_1$ induce the same measure: \[ \lambda\circ \varphi_0 = \lambda\circ \varphi_1 = 1/2( \delta_{\mathcal{U}_0} + \delta_{\mathcal{U}_1}), \]
	but $\Vdash_\mathbb{M} \dot{\varphi_0} \ne \dot{\varphi_1}$.
\end{exa}

The above example shows that we have to be careful if we want to deduce something about the relationship of names for ultrafilters by looking at the induced measures.

Perhaps the most important Boolean algebra to study in this context is the Cantor algebra, as its ultrafilters can be interpreted as reals. 

\begin{exa}\label{cantor} Consider the Cantor algebra $\mathbb{C} = \mathrm{Clop}(2^\omega)$. 
	Let $\varphi\colon \mathbb{C} \to \mathbb{M}$ be a homomorphism sending clopens to its equivalence classes in $\mathbb{M}$. Then
	$\dot{\varphi}$ is a name for the generic random real (see \cite[Definition 3.9]{Kunen-cohen-and-random}). In general, suppose that $\varphi\colon \mathbb{C} \to \mathbb{M}_\kappa$ is a homomorphism and $\mu = \lambda_\kappa \circ \varphi$. If
	$\mu$ is atomless, then $\dot{\varphi}$ is a name for a 'new' real in the following sense: for each $x\in 2^\omega$ there is $p\in \mathbb{M}_\kappa$ such that $p\Vdash x \ne \dot{\varphi}$. Indeed, take $C\in \mathbb{C}$ such that $x\notin C$
	and $\mu(C)>0$. Then $\varphi(C) \Vdash x \ne \dot{\varphi}$.
\end{exa}

\begin{rem}
	Recall that if $\dot{r}$ is a random real then it avoids measure zero sets coded in the ground model (see \cite[Lemma 3.8]{Kunen-cohen-and-random}). An example of application of this fact is the following. If ($\alpha$) is a Borel property of real numbers such that $\{x\in
	2^\omega\colon x $  has ($\alpha$)$\}$ is of full measure, then $\Vdash_{\mathbb{M}_\kappa} \dot{r}$ possess ($\alpha$), where $\dot{r}$ is the generic random real. So, e.g., Strong Law of Large Numbers implies that $\Vdash_{\mathbb{M}_\kappa}$ $\dot{r}$ has asymptotic density 1/2 (seen
	as a subset of $\omega$). 
	Using measures we can find names for new reals which have some additional properties. Let $(\alpha)$ be a Borel property of reals such that $F = \{x \in 2^\omega \colon x\mbox{ has }(\alpha)\}$ is not scattered. Then there is an atomless measure $\mu$
	on $\mathbb{C}$ such that $\hat{\mu}(F)=1$. By Proposition \ref{Maharam-conclusion}, there is $\kappa$
	and a homomorphism $\varphi\colon \mathbb{C} \to \mathbb{M}_\kappa$ such that $\mu = \lambda_\kappa \circ \varphi$. Then $\dot{\varphi}$ is a name for a new real (see Example \ref{cantor}) such that 
	$ \Vdash_{\mathbb{M}_{\kappa}} \dot{\varphi} \mbox{ has }(\alpha)$. \end{rem}

The next example will concern the algebra $\mathcal{P}(\omega)$. The Boolean algebra of those subsets of $\omega$ which belongs to the ground model is a subalgebra of $\mathcal{P}(\omega)$ in the forcing extension (which is a proper subalgebra if the
forcing adds new reals). The Stone spaces of such Boolean algebras of 'old' subsets of $\omega$ were considered e.g. in \cite{DamLub} and in the context of random forcing in \cite{DowFremlin}. 
The measures on $\mathcal{P}(\omega)$ were studied e.g. in \cite{FremlinTalagrand}. Note that there are many atomless measures on $\mathcal{P}(\omega)$, although none of them can be strictly positive. A typical example of such a measure is of the
form $\mu(A) = \lim_{n\to \mathcal{U}} |A \cap n|/ n$, where $\mathcal{U}$ is an ultrafilter on $\mathcal{P}(\omega)$ (see \cite{PlebanekRyll}). Such measure extends the asymptotic density. 

Recall that an ultrafilter $\mathcal{U}$ on $\mathcal{P}(\omega)$ is a P-point if $\mathcal{U}$ is $\sigma$-closed with respect to $\subseteq^*$, i.e. inclusion modulo finite sets. Notice that this definition makes sense also for subalgebras of
$\mathcal{P}(\omega)$, in particular of $\mathcal{P}(\omega) \cap V$ if we work in a forcing extension. The following fact says that the new ultrafilters on $\mathcal{P}(\omega) \cap V$ cannot be P-points. 

\begin{prop}\label{ppoint} Suppose that $\varphi\colon \mathcal{P}(\omega) \to \mathbb{M}$ is such that $\mu = \lambda \circ \varphi$ is atomless. Then $\Vdash_\mathbb{M} \dot{\varphi}$ is not a P-point on $\mathcal{P}(\omega) \cap V$. 
\end{prop}
\begin{proof}
	 Denote $\mathbb{A} = \mathcal{P}(\omega)$. 
	 Since
	 $\mu$ is atomless for every $n$ we can find a maximal finite antichain $\mathcal{A}_n$ in $\mathbb{A}$ such that $\mu(A)<1/n$ for each $A\in \mathcal{A}_n$. Let $G$ be a generic  filter on $\mathbb{M}$. From this point on we will work in $V[G]$,
	 understanding $\mathbb{A}$ as $\mathcal{P}(\omega)\cap V$. Notice that for each $n$ there is exactly one $A_n \in \mathcal{A}_n$ such that $\varphi(A_n) \in G$. 
	 Then $(A_n)$ is a sequence of elements of $\dot{\varphi}$. But if $A\in \mathbb{A}$ is such that $A\subseteq^* A_n$ for each $n$, then $\mu(A) =0$ (because $\mu$ is atomless). Hence $\varphi(A)=0$  and so  
	 $ 1 \Vdash A\notin \dot{\varphi}$. Thus, $(A_n)$ witnesses that $\dot{\varphi}$ is not a P-point.
\end{proof}

Note that Kunen in \cite{Kunen-special-points} used measures on $\omega$ in a slightly different way to show that in the random model there are no selective ultrafilters.

Now we will prove a result showing that in the forcing extension by $\mathbb{M}_{\omega_2}$ over a model satisfying $\mathsf{GCH}$ (so in the classical random model) the sets of reals of size $\mathfrak{c}$ have to be in a sense homogeneous. 

\begin{thm}\label{Compatible} Assume $GCH$. Let $(\varphi_\alpha)_{\alpha<\omega_2}$ be a family of homomorphisms $\varphi_\alpha\colon \mathbb{C} \to \mathbb{M}_{\omega_2}$. Then there are $\alpha<\beta<\omega_2$ and an automorphism $\Phi\colon
	\mathbb{M}_{\omega_2} \to
	\mathbb{M}_{\omega_2}$ such that $\Phi\circ \varphi_\alpha(C) = \varphi_\beta(C)$ and $\Phi\circ \varphi_\beta(C) = \varphi_\alpha(C)$ for every $C\in \mathbb{C}$.
\end{thm}

\begin{proof} 
	First, fix an independent family $\{C_n\colon n\in\omega\} \subseteq \mathbb{C}$ generating $\mathbb{C}$. We say that $A \in \mathbb{C}$ is \emph{a chunk} if it is of the form \[ A = (C_{i_0} \land \dots \land C_{i_k} ) \land ( C^c_{j_0} \land \dots \land
	C^c_{j_l} ) \]
	for sequences $(i_n)_{n\leq k}$, $(j_n)_{n\leq l}$ of integers.


We say that $\alpha$ and $\beta$ are \emph{symmetric} if 
\[ \varphi_\alpha(A) \land \varphi_\beta(B) = 0 \iff \varphi_\beta(A) \land \varphi_\alpha(B) =0 \]
for each chunks $A$ and $B$.
\medskip

\textbf{Claim.} There are distinct $\alpha$, $\beta < \omega_2$ which are symmetric. 
\medskip

Suppose the opposite. Then for each $\alpha<\beta$ we can choose chunks $A$, $B$ and $i\in \{0,1\}$ such that
\[ i=0 \mbox{ and } \varphi_\alpha(A) \land \varphi_\beta(B) =0 \mbox{ and } \varphi_\beta(A) \land \varphi_\alpha(B) \ne 0 \]
or
\[ i=1 \mbox{ and } \varphi_\alpha(A) \land \varphi_\beta(B) \ne 0 \mbox{ and } \varphi_\beta(A) \land \varphi_\alpha(B) = 0. \]

Since $\mathbb{C}\times \mathbb{C} \times \{0,1\}$ is countable, by GCH and Erd\H{o}s-Rado theorem $\omega_2 = 2^{2^\omega} \rightarrow (\omega^+)^2_\omega$, we could find an uncountable set $\Lambda \subseteq \omega_2$ such that 
for each $\alpha, \beta \in \Lambda$ we have assigned the same configuration of chunks $A$, $B$ and $i\in \{0,1\}$. We may assume without loss of generality that $\Lambda = \omega_1$, $i=0$ and so for each $\alpha<\beta<\omega_1$ we have
\[ \varphi_\alpha(A) \land \varphi_\beta(B) = 0 \mbox{ and } \varphi_\beta(A) \land \varphi_\alpha(B) \ne 0. \]
Let $D = \bigvee_{\alpha<\omega_1} \varphi_\alpha(A)$. There is $\gamma<\omega_1$ such that $\bigvee_{\alpha<\gamma} \varphi_\alpha(A) = D$. Then $D \land \varphi_\gamma(B) = 0$. Let $\beta>\gamma$. Then $ \varphi_\beta(A) \land \varphi_\gamma(B)  \ne 0$
but $\varphi_\beta(A) \leq D$, a contradiction.
\bigskip

So, let $\alpha<\beta<\omega_2$ be symmetric. Define $\Phi$ on $\mathcal{A} = \{\varphi_i(C_n)\colon n\in \omega, i \in \{\alpha,\beta\}\}$ by setting $\Phi(\varphi_\alpha(C_n)) =  \varphi_\beta(C_n)$ and $\Phi(\varphi_\beta(C_n)) =
\varphi_\alpha(C_n)$. Then $\Phi$ can be extended to an automorphism of the Boolean algebra generated by $\mathcal{A}$. Indeed, by Sikorski's Extension Criterion (see \cite[Proposition 5.6]{BAhandbook}) it is enough to check that 
\[ \varphi_\alpha(A) \land \varphi_\beta(B) = 0  \iff (\Phi(\varphi_\alpha(A)) \land \Phi(\varphi_{\beta}(B)) = 0. \]
	But since $\Phi\circ \varphi_\alpha = \varphi_\beta$ and $\Phi\circ \varphi_\beta = \varphi_\alpha$ the above condition translates to symmetry of $\alpha$ and $\beta$.

	Then $\Phi$ can be extended to an automorphism of the whole $\mathbb{M}_{\omega_2}$ (see \cite[proof of (a) in Theorem 383E]{Fremlin-MT3}, \cite[Theorem 333C(b)]{Fremlin-MT3}).
\end{proof}

As a corollary we will reprove one of the famous results of Kunen. We will show that in the random model there are no towers of size $\omega_2$. In fact, the Kunen's original theorem is about the Cohen model. Kunen's proof seems to belong to the oral
part of the set theoretic tradition and we are not aware of any literature containing it (apart perhaps of \cite{Kunen-Phd}). 

\begin{thm}[Kunen]\label{towers} Let $V$ be the model of $GCH$ and let $G$ be the $\mathbb{M}_{\omega_2}$-generic. Then, in $V[G]$, there is no well-ordered chain of size $\omega_2$ in $\mathcal{P}(\omega)/Fin$.
\end{thm}
\begin{proof} Suppose that there is a strictly $\subseteq^*$-increasing chain of size $\omega_2$ in $V[G]$. Then there is a sequence $(\varphi_\alpha)_{\alpha<\omega_2}$ of homomorphisms $\varphi_\alpha\colon \mathbb{C} \to \mathbb{M}_{\omega_2}$ such that
\[ \Vdash_{\mathbb{M}_{\omega_2}} (\dot{\varphi}_\alpha)_{\alpha<\omega_2} \mbox{ is a }\subseteq^*\mbox{-chain of subsets of }\omega. \]

We will work in $V$.  By Theorem \ref{Compatible} there are $\alpha<\beta<\omega_2$ and an automorphism $\Phi\colon \mathbb{M}_{\omega_2} \to \mathbb{M}_{\omega_2}$ such that $\Phi\circ \varphi_\alpha = \varphi_\beta$ and $\Phi\circ \varphi_\beta = \varphi_\alpha$.
 Then 
\[ \Phi[\mathbb{M}_{\omega_2}] \Vdash \Phi\circ \dot{\varphi}_\alpha \subseteq^* \Phi\circ \dot{\varphi}_\beta, \]
and so 
\[ \mathbb{M}_{\omega_2} = \Phi[\mathbb{M}_{\omega_2}] \Vdash \dot{\varphi}_\beta \subseteq^* \dot{\varphi}_\alpha, \]
which clearly leads to the contradiction with the assumption that the chain is strictly increasing.
\end{proof}

	
	
%

\begin{rem} Every $\mathbb{M}$-name for an ultrafilter on $\mathbb{M}$ itself can be seen as a homomorphism $\varphi\colon \mathbb{M} \to \mathbb{M}$. The family of all such homomorphisms with the composition is a semi-group. So, in $V[G]$, where $G$
	is an $\mathbb{M}$-generic, we have a semi-group structure on the Stone space of $\mathbb{M} \cap V$. It seems that this structure has no natural counterpart in the Stone space of $\mathbb{M}$ in the ground model. In the next section we will see that
	this structure with a little bit of ergodic theory may help us to prove some theorems about the Stone space of $\mathbb{M}\cap V$.
\end{rem}

\section{Measure Recognition Problems in forcing context}\label{mrp}

In this section we overview some of the properties used in the Measure Recognition Problems mentioned in the introduction. We will show that those properties have characterization in terms of forcing with measure algebras.

\subsection{Boolean algebras supporting measures}
The question about a characterization of Boolean algebras supporting measures has a long history. In case we ask this question in the context of $\sigma$-additive measures, it leads to the famous Maharam's Problem. The combinatorial charaterization for finitely
additive measures is much less mysterious; it was formulated  by Kelley in \cite{Kelley}. 

In \cite{Kamburelis} Kamburelis proved an elegant characterization of Boolean algebras supporting measures in the forcing terminology (Theorem \ref{Kamburelis} below). If we look at the names for ultrafilters as homomorphisms, the proof
almost trivializes provided we use a well-known fact from the theory of dynamical systems.

Let $(X, \Sigma, \mu)$ be a measure space. We say that $T\colon X \to X$ is a \emph{measure-preserving transformation} if $\mu(T[A]) = \mu(A)$ for each measurable $A$. 

A measure-preserving transformation $T$ is called \emph{ergodic} if whenever $A$ is invariant, i.e. $T^{-1}[A]=A$, then $\mu(A)=0$ or $\mu(A)=1$. Note that this is equivalent to saying that whenever $A, B$ are $\mu$-positive then  
 $\mu(T^{-n}[A] \cap B)>0$ for some $n$. It follows from the fact that $\bigcup_n T^{-n}[A]$ is invariant and so $\mu(\bigcup_n T^{-n}[A])=1$.

An example of an ergodic transformation of $([0,1], \mathrm{Bor}([0,1]), \lambda)$ is $T(x) = x + r$, where $r$ is irrational and $+$ is taken modulo $1$. We will use the fact that for each infinite $\kappa$ there is an ergodic transformation of
$([0,1]^\kappa, \mathrm{Bor}([0,1]^\kappa), \lambda_\kappa)$. 

\begin{prop}\label{ergodic} \cite[Section 1.4]{Choksi} Let $\kappa$ be an infinite cardinal number. There is a metric automorphism $\varphi\colon \mathbb{M}_\kappa \to \mathbb{M}_\kappa$ such that whenever $A, B \in \mathbb{M}^+_\kappa$, there is $n\in \omega$ such that $\varphi^n(A) \cap B \ne 0$.
\end{prop}

\begin{proof} Let $X = [0,1]^\kappa$. Consider $X^\mathbb{Z}$ with the product measure $\mu$. Let $T\colon X^\mathbb{Z} \to X^\mathbb{Z}$ be the Bernoulli shift, i.e. $T(x)(n) = x(n+1)$. Then, for every measurable $A, B \subseteq X^\mathbb{Z}$ we
	have $\lim_{n\to \infty} \mu(T^{-n}[A] \cap B) = \mu(A)\cdot\mu(B) > 0$ (i.e. Bernoulli shift is mixing, see e.g. \cite[Chapter 3]{Shields}) and so $T$ is ergodic.
	
	Define a metric automorphism $\varphi$ of $\mathrm{Bor}(X^\mathbb{Z})/_{\mu = 0}$ by $\varphi(A) = T^{-1}[A]$ and notice that, by Maharam's theorem, $(\mathrm{Bor}(X^\mathbb{Z})/_{\mu = 0}, \mu)$ is, in fact, isomorphic to $(\mathbb{M}_\kappa,
	\lambda_\kappa)$ and so we may think that $\varphi$ is an automorphism of $\mathbb{M}_\kappa$. Clearly, $\varphi$ satisfies the desired property.
\end{proof}

\begin{thm}\label{Kamburelis}\cite[Proposition 3.7]{Kamburelis} Let $\mathbb{A}$ be a Boolean algebra. Then the following are equivalent
	\begin{itemize}
		\item $\mathbb{A}$ supports a (strictly positive) measure,
		\item there is $\kappa$ such that $1 \Vdash_{\mathbb{M}_\kappa} \mathbb{A}$ is $\sigma$-centered.
	\end{itemize}
\end{thm}

\begin{proof}
	Assume that $\mathbb{A}$ supports a measure $\mu$. Then, by Proposition \ref{Maharam-conclusion}, there is a metric homomorphism $\psi\colon \mathbb{A} \to \mathbb{M}_\kappa$ for some $\kappa$. 	   Let $\varphi\colon
	\mathbb{M}_\kappa \to \mathbb{M}_\kappa$ be as in Proposition \ref{ergodic}. Denote $\rho_n = \varphi^n \circ \psi$. We claim that \[ 1 \Vdash_{\mathbb{M}_\kappa}  \mathbb{A}^+ = \bigcup_n \dot{\rho}_n. \]
	Indeed, let $A\in \mathbb{A}^+$ and let $p\in \mathbb{M}^+_\kappa$. There is $n$ such that $q = (\varphi^n( \psi(A) ) \cap p) \ne 0$. But then $q \Vdash A \in \dot{\rho_n}$. By the density argument we are done.

	Now suppose that $1 \Vdash_{\mathbb{M}_\kappa} \mathbb{A}$ is $\sigma$-centered. Then there is a family $\{\dot{\rho}_n \colon n\in \omega\}$ of $\mathbb{M}_\kappa$-names for ultrafilters on $\mathbb{A}$ 
	such that \[ 1\Vdash_{\mathbb{M}_\kappa} \mathbb{A}^+ = \bigcup_n \dot{\rho}_n. \]
	For every $n$ let $\mu_n = \lambda \circ \rho_n$. Then $\mu = \sum_n \mu_n/2^n$ is strictly positive on $\mathbb{A}$.
	Otherwise, there is $A\in \mathbb{A}^+$ such that $\mu_n(A)=0$ for each $n$. Hence, $\rho_n(A) = 0$ for every $n$ and so for each $n$ we would have $1 \Vdash_{\mathbb{M}_\kappa} A\notin \dot{\rho}_n$, a contradiction.
\end{proof}

As a corollary we get a simple proof of the folklore fact saying that the measure algebra $\mathbb{M}$ becomes $\sigma$-centered in the forcing extension by $\mathbb{M}$. This is formulated e.g. in \cite[Proposition
3.2.11]{Bartoszynski}
without a (direct) proof. We show that in fact a dense subset of the Stone space of the 'old' $\mathbb{M}$ in the forcing extension may be induced by a single object.

\begin{prop} $\Vdash_\mathbb{M} \mathbb{M} \mbox{ is }\sigma\mbox{-centered}$.
\end{prop}

\begin{proof}
	The proof is analogous to the first part of the proof of Theorem \ref{Kamburelis}. In this case we do not need Maharam's theorem; $\psi$ is just the identity.
\end{proof}

\subsection{Boolean algebras carrying atomless measures} 

The following fact is a generalization of Example \ref{old}.

\begin{prop}\label{purely-atomic} Let $\varphi\colon \mathbb{A} \to \mathbb{M}_\kappa$. Assume that the measure $\mu = \lambda_\kappa \circ \varphi$ is purely atomic, i.e. $\mu = \sum_n a_n \delta_{\mathcal{V}_n}$ for some sequence $(a_n)$ of
	positive real numbers such that $\sum_{n} a_n = 1$, and a sequence $(\mathcal{V}_n)$ of ultrafilters
on $\mathbb{A}$. Then, $1 \Vdash_{\mathbb{M}_\kappa} \exists n\  \dot{\varphi} = \mathcal{V}_n$. \end{prop}

\begin{proof} For $n\in \omega$ define $p_n = \bigwedge\{ \varphi(V)\colon V\in \mathcal{V}_n\}$. It is plain to check that $p_n \Vdash \dot{\varphi} = \mathcal{V}_n$. Notice also that $(p_n)$ is an antichain. If $n\ne m$ then there is $V\in \mathcal{V}_n \setminus \mathcal{V}_m$ and so $p_n \leq \varphi(V)$ and $p_m \leq \varphi(V)^c$. 

	To finish the proof we will show that $(p_n)$ is a maximal antichain (then the conclusion will follow from the standard density argument). First, since $\mathbb{M}_\kappa$ is ccc and each $\mathcal{V}_n$ is
	closed under intersections, for every $n$ there is a
	decreasing sequence $(V_k)$ in $\mathcal{V}_n$ such that $p_n = \bigwedge_k \varphi(V_k)$. So, $\lambda_\kappa(p_n)\geq a_n$, by the continuity of $\lambda_\kappa$. Thus, $\sum_n \lambda_\kappa(p_n) \geq \sum_n a_n = 1$ and so $(p_n)$ is
	maximal. 
\end{proof}

In particular, if the measure induced by a name $\dot{\varphi}$ is purely atomic, then $1 \Vdash_{\mathbb{M}_\kappa} \dot{\varphi}$ is in the ground model. At this point, we may formulate an easy corollary. Recall that a Boolean algebra is
\emph{superatomic} if each of its subalgebras contains an atom, equivalently, if its Stone space is scattered, see Introduction.

\begin{cor}\label{super} Let $\mathbb{A}$ be a Boolean algebra. Then $\mathbb{A}$ is superatomic if and only if $1\Vdash_{\mathbb{M}_\kappa} \dot{\mathrm{St}}(\mathbb{A}) = \mathrm{St}(\mathbb{A})$ for every $\kappa$.
\end{cor}
\begin{proof} Use the fact that a Boolean algebra is superatomic if and only if all the measures it carries are purely atomic (see \cite{Semadeni}) and Proposition \ref{purely-atomic}.
\end{proof}

\begin{rem} 
	We get the above corollary 'for free' from Proposition \ref{purely-atomic}. In fact, it can be generalized to other forcing notions. Recall also Mazurkiewicz-Sierpi\'nski theorem which says that compact scattered spaces are classifiable up to
	homeomorphism by
	$\omega_1 \times \omega$ (see \cite{Mazurkiewicz} and \cite{Day}) and in this sense they are absolute objects.
\end{rem}

\subsection{Boolean algebras carrying only separable measures}

It is well known that every real added by forcing with $\mathbb{M}_\kappa$ can be added in fact added by forcing with $\mathbb{M}$. We will see that
we may characterize those names for ultrafilters which can be added by forcing with $\mathbb{M}$.



\begin{thm}\label{separable} Assume $\mathbb{A}$ is a Boolean algebra and let $\varphi\colon \mathbb{A} \to \mathbb{M}_\kappa$ be a homomorphism such that $\lambda_\kappa \circ \varphi$ is separable. Then there is $\mathbb{M}'\subseteq
\mathbb{M}_\kappa$ isomorphic to $\mathbb{M}$  such that $\varphi[\mathbb{A}] \subseteq \mathbb{M}'$. In particular $\dot{\varphi}$ is an $\mathbb{M}'$-name and so it can be added by forcing with $\mathbb{M}$. \end{thm}
\begin{proof}	Assume $\mathbb{A}$ and $\varphi$ are as above. For simplicity we will assume that $1 \Vdash_{\mathbb{M}_\kappa} \dot{\varphi} \notin V$ and so that $\mu$ is atomless.
	
First, notice that $(\mathbb{A}/_{\mu=0}, \mu)$ and $(\varphi[\mathbb{A}], \lambda_\kappa)$ are metrically isomorphic. Indeed, $\varphi\colon \mathbb{A} \to \varphi[\mathbb{A}]$ is a metric epimorphism, whose kernel coincides with the family of
elements of $\mathbb{A}$ which are $\mu$-null. Hence, by the assumption on $\mu$, the measure $\lambda_\kappa$ is separable on $\varphi[\mathbb{A}]$ and so $(\varphi[\mathbb{A}], d_{\lambda_\kappa})$ is a separable metric space. Let $\mathbb{M}'$ be
the closure of $\varphi[\mathbb{A}]$ in $(\mathbb{M}_\kappa, d_{\lambda_\kappa})$. Then $\mathbb{M}'$ is a Boolean subalgebra of $\mathbb{M}_\kappa$ and $(\mathbb{M}',d_{\lambda_\kappa})$ is still separable. But $\mathbb{M}'$ is complete as a closed subspace of $\mathbb{M}_\kappa$ and so, by
	\cite[Proposition 2]{Prikry}, it is
	also complete as a Boolean algebra. Thus, by Maharam's theorem, it is isomorphic to $\mathbb{M}$. 
\end{proof}

\begin{prop}\label{hereditary} A Boolean algebra $\mathbb{A}$ supports a separable measure if and only if $\mathbb{A}$ can be embedded in $\mathbb{M}$.  
\end{prop}

\begin{proof} The forward implication is just Proposition \ref{Maharam-conclusion} and the reverse implication follows from the fact that subspaces of separable metric spaces are separable.
\end{proof}

Recall that a Boolean algebra is measure separable if it carries only separable measure (see also Introduction).

\begin{thm} \label{measure-separable} Let $\mathbb{A}$ be a Boolean algebra. Then the following conditions are equivalent:
	\begin{itemize}
		\item $\mathbb{A}$ is measure separable,
		\item for every   $\kappa\geq \omega$ and every $\mathbb{M}_\kappa$-name $\dot{\varphi}$ for an ultrafilter on $\mathbb{A}$ there is a subalgebra $\mathbb{M}'\subseteq \mathbb{M}_\kappa$ isomorphic to $\mathbb{M}$ and such that $\dot{\varphi}$ is an $\mathbb{M}'$-name for an ultrafilter on $\mathbb{A}$. 
	\end{itemize}
\end{thm}
\begin{proof}
	If $\mathbb{A}$ is measure separable and $\varphi\colon \mathbb{A} \to \mathbb{M}_\kappa$ is a homomorphism, then $\lambda_\kappa \circ \varphi$ is separable and we can use Proposition \ref{separable}.

	Suppose now that $\mathbb{A}$ is not a measure separable Boolean algebra and let $\mu$ be a non-separable measure on $\mathbb{A}$. Let $K$ be the Stone space of $\mathbb{A}$. We may assume without loss of generality that
	$\mathrm{Bor}(K)/_{\hat{\mu}=0}$ is metrically isomorphic to $\mathbb{M}_\kappa$ for some
	$\kappa>\omega$. Assume that $\varphi' \colon \mathrm{Bor}(K)_{/\hat{\mu}=0} \to \mathbb{M}_\kappa$ witnesses this isomorphism. Let $\varphi\colon \mathbb{A} \to \mathbb{M}_\kappa$ be defined by $\varphi(A) = \varphi'([A]_{\hat{\mu}=0})$. Then
	$\varphi[\mathbb{A}] = \mathbb{M}_\kappa$ and so $\dot{\varphi}$ is not a $\mathbb{M}'$-name for any proper subalgebra of $\mathbb{M}_\kappa$.

	\end{proof}


The above theorem says that ultrafilters on a measure separable Boolean algebras which can be added by many random reals can be added also by a single random real. This does not mean that the Stone space of such Boolean algebra in the universe extended by
adding many random reals is the same as in the universe extended by a single random real. E.g. the Cantor algebra is measure separable, as countable Boolean algebra, but the cardinality of its Stone space clearly depends on how many random reals we add.

Notice that whenever $\mathbb{A}$ contains an uncountable independent
subfamily then it carries a measure which is non separable. Actually, under $\mathsf{MA}_{\omega_1}$ the converse implication also holds true (see \cite{Fremlin}) and so we have an immediate corollary.

\begin{cor} Assume $\mathsf{MA}_{\omega_1}$. If $\mathbb{A}$ is a Boolean algebra without an uncountable
	independent subfamily, then each ultrafilter on $\mathbb{A}$ added by random reals can be added by a single random real.
\end{cor}

It is well known that consistently there are Boolean algebras without an uncountable independent family supporting non-separable measures. Thus, even a small Boolean algebra may have $\mathbb{M}_{\omega_1}$-names for ultrafilters which cannot be add
by a single random real. In fact, under $\mathsf{CH}$, there is a Boolean algebra with a much stronger property. The following theorem was proved in \cite{Kunen}.

\begin{thm}[Kunen]\label{Lspace} Assume $\mathsf{CH}$. There is a Boolean algebra $\mathbb{A}$ (with the Stone space $K$) and a measure $\mu$ on $\mathbb{A}$ such that 
	\begin{itemize}
		\item $\mathrm{Bor}(K)_{/ \hat{\mu}=0}$ is metrically isomorphic to $\mathbb{M}_{\omega_1}$ (in particular, $\hat{\mu}_{|B}$ is non-separable for every $\hat{\mu}$-positive $B\subseteq K$). 
		\item $F\subseteq K$ is nowhere dense if and only if $\hat{\mu}(F)=0$,
		\item $\mathbb{A}$ does not contain an uncountable independent family.
	\end{itemize}
\end{thm}

\begin{cor}\label{pokunen} Assume $\mathsf{CH}$. There is a Boolean algebra $\mathbb{A}$ which supports a measure and such that each separable measure is supported on a nowhere dense subset of its Stone space.
\end{cor}

\begin{proof} 
	Let $\mathbb{A}$, $K$, $\mu$ be as in  Theorem \ref{Lspace}. Let $\nu$ be a separable measure carried by $\mathbb{A}$. 
	
	We will work in $K$. First, we will show that $\hat{\mu} \perp \hat{\nu}$. By Jordan decomposition $\hat{\nu} = \hat{\nu_0} + \hat{\nu_1}$, where $\hat{\nu_0} \perp \hat{\mu}$ and $\hat{\nu_1}\ll \hat{\mu}$. Since $\nu$ is separable, both
	$\hat{\nu}_0$ and $\hat{\nu}_1$ are separable, too.  By Radon-Nikodym theorem $\hat{\nu}_1 = \int f \ d\hat{\mu}$ for a measurable function $f\colon K\to \mathbb{R}$. Suppose for the contradiction that $\hat{\nu}_1(K)>0$. Then the function $f$ is not almost everywhere $0$ and so there is a $\hat{\mu}$-positive $B\subseteq K$ and
	$\varepsilon>0$ such that $f(x)>\varepsilon$ for each $x\in B$. But then $\hat{\nu}_1(A)>\varepsilon \hat{\mu}(A)$
	for each $A\subseteq B$ and so $\hat{\nu_1}_{| B}$ is not separable. Thus, $\hat{\nu}_1$ is not separable, a contradiction. So $\hat{\nu}_1 = 0$ and so $\hat{\mu} \perp \hat{\nu}$.

	So, there is $F\subseteq K$ such that $\hat{\mu}(F)=0$ and $\hat{\nu}(F)=1$. But then $F$ is nowhere dense and $\hat{\nu}$ is supported by $\overline{F}$ which is still nowhere dense.
\end{proof}

So, forcing with $\mathbb{M}$ adds ultrafilters on $\mathbb{A}$ which can be quite precisely 'localized': for each $\mathbb{M}$-name $\dot{\mathcal{U}}$ for an ultrafilter on $\mathbb{A}$ we can find a nowhere dense set $D\subseteq K$ in the ground model such that
$\Vdash_\mathbb{M} \dot{\mathcal{U}} \in D$. 

\section{Non-trivial convergent sequences of homomorphisms}\label{nontri}

In this section we will deal with homomorphisms into measure algebras in isolation with its forcing interpretations. First, we observe that we may use the fact that measure algebras are metrizable in a natural way to define notions of convergence of
homomorphisms.

\begin{df} Let $\mathbb{A}$ be a Boolean algebra and let $(\varphi_n)$ be a sequence of homomorphisms $\varphi_n\colon \mathbb{A} \to \mathbb{M}_\kappa$, where $\kappa$ is a cardinal number. We say that
	\begin{itemize}
		\item $(\varphi_n)$ converges to $\varphi$ \emph{pointwise} if $d_{\lambda_\kappa}(\varphi_n(A), \varphi(A))$ converges to $0$ for each $A\in \mathbb{A}$.
		\item $(\varphi_n)$ converges to $\varphi$ \emph{uniformly} if for each $\varepsilon>0$ there is $N$ such that for each $n>N$ and each $A\in \mathbb{A}$ we have $d_{\lambda_\kappa}(\varphi_n(A),\varphi(A)) <\varepsilon.$
	\end{itemize}
\end{df}

The following simple fact implies in particular that the Stone topology is a pointwise convergence topology.

\begin{prop}\label{nontrivial}
\label{uniform-conv}
Let $(\varphi_n)$ be a sequence of homomorphisms $\varphi_n\colon \mathbb{A} \to \{0,1\}$. For each $n$ define an ultrafilter $U_n$ on $\mathbb{A}$ by $U_n = \varphi^{-1}_n[\{1\}]$. Then
\begin{itemize}
	\item $(\varphi_n)$ converges pointwise if and only if $(U_n)$ converges (in the Stone topology),
	\item $(\varphi_n)$ converges uniformly if and only if $(U_n)$ converges trivially.
\end{itemize}
\end{prop}

To some extend this proposition can be adapted to a more general situation of sequence of homomorphisms to measure algebras and convergence in the random models (see \cite{PbnSobota}).

Motivated by Proposition \ref{nontrivial} we say that a sequence of homomorphisms into a measure algebra \emph{converges non-trivially}  if it converges pointwise but not uniformly. We will show that for every infinite Boolean algebra $\mathbb{A}$ there are
nontrivial convergent sequences into measure algebras. In the proof we will treat separately the case when $\mathbb{A}$ contains the Cantor algebra and when it is superatomic.




\begin{prop}
\label{point-not-uni-withC}
Let  \(\mathbb{A}\) be a Boolean algebra containing the Cantor subalgebra. Then there is sequence $(\varphi_{n})$ of homomorphisms $\varphi_n\colon \mathbb{A} \to \mathbb{M}$ convergent pointwise but not uniformly.
\end{prop}

\begin{proof}
	We will treat the elements of the Cantor algebra $\mathbb{C}$ (embedded in $\mathbb{A}$) as clopen subsets of $2^\omega$ and the elements of $\mathbb{M}$ as Borel subsets of $2^\omega$. So, we define $\bar{\varphi}\colon \mathbb{C} \to \mathbb{M}$ by $\bar{\varphi}(C) =
	C$ understanding $C$ once as a clopen set, once as a representative of an element of $\mathbb{M}$. By Sikorski's Extension Theorem (\cite[Theorem 5.9]{BAhandbook}), since \(\mathbb{M}\) is complete, there isa homomorphism \(\varphi\) extending $\bar{\varphi}$. 

For $n\in \omega$ define  $f_{n}\colon 2^\omega \to 2^\omega$  by
\[f_{n}(a)=a +_{2} e_{n}\]
where \(+_{2}\) denotes addition modulo \(2\) and 
\[e_{n}(i)=
\begin{cases}
1 &  \text{for } i = n \\
0  & \text {for } i \neq n.
\end{cases}\]

Now for $n\in \omega$ define $\psi_n \colon \mathbb{M}\to \mathbb{M}$ by $\psi_{n}(B)=f_{n}[B]$ (again, identifying elements of $\mathbb{M}$ with its Borel representatives). Finally, let $\varphi_n = \psi_n \circ \varphi$.  We claim that the
sequence $(\varphi_{n})$ is the sought one.

First, denote $C_n = \{x\in 2^\omega\colon x(n)=1\}$ and notice that $\varphi_n(C_n) = C_n^c = \varphi(C_n)^c$. Thus, $d_\lambda(\varphi_{n+1}(C_n), \varphi_n(C_n))=1$ and so $(\varphi_n)$ does not converge uniformly.

To see that $(\varphi_n)$ converges pointwise, fix $\varepsilon>0$ and $A\in \mathbb{A}$. Let $C\in \mathbb{C}$ be such that $\lambda(C \triangle A)<\varepsilon/2$. Notice that $\varphi_n(C) = C$ for $n$'s big enough, as elements of $\mathbb{C}$ 
are finite Boolean combinations of $C_n$'s. Since $\varphi_n$ is a metric isomorphism, $\lambda(\varphi_n(A) \triangle \varphi_n(C))<\varepsilon/2$ and so $\lambda(\varphi_n(A) \triangle C)<\varepsilon/2$ for $n$'s big enough. Hence,
$\lambda(\varphi_n(A) \triangle A) < \varepsilon$ for $n$'s big enough and so $\varphi_n$ converges to identity.
\end{proof}

Now we will show that the conclusion of Proposition \ref{point-not-uni-withC} holds true also for Boolean algebras which does not contain the Cantor algebra. Those Boolean algebras are exactly the superatomic Boolean algebras mentioned in the previous
section. The Stone spaces of superatomic algebras are scattered spaces. It is well-known and widely used that scattered spaces contain nontrivial convergent sequences. However, we will enclose a proof of it as it is difficult to find it in the
literature (at least in the easily accessible one). The following proof is based on the proof contained in \cite{Sobota-Phd}. 


\begin{prop} 
If \(K\) is an infinite scattered compact space then \(K\) contains a nontrivial convergent sequence.
\end{prop}

\begin{proof} Let \(K\) be an infinite scattered compact space.  Let \(F_{0}=K\) and let $x_0$ be an isolated point of $F_0$. Suppose that for $n\in \omega$ we have constructed $F_n$ and $x_n\in F_n$ and define $F_{n+1} = F_n \setminus \{x_n\}$ and $x_{n+1}$ as an isolated point of $F_{n+1}$. 
Define 
\[Y= \{x_n\colon n\in \omega\}.\]

Then $Y$ is an infinite discrete subset of $K$ and so it is not closed and  \(\text{bd}Y \neq
\emptyset\). Hence, there is an isolated point \(x \in \text{bd}Y\). We claim that \(x\) is a limit of a sequence from \(Y\). 

Indeed, \(\{x\}\) and \(\text{bd}Y \setminus \{x\}\) are closed and thus there are open, disjoint  \(U, V \subseteq K\) such that \(\{x\} \subseteq U\) and \(\text{bd}Y \setminus \{x\} \subseteq V\). Notice that since $Y$ is discrete \(\overline{U \cap Y}=(U \cap
Y) \cup \{x\}\). Enumerate $U\cap Y = \{y_n\colon n\in \omega\}$. 

We claim that $(y_n)$ converges to $x$. Indeed, let $W$ be an open neighbourhood of $x$ and denote $Z = U\cap Y \setminus W$.  Then $\overline{Z} \cap \mathrm{bd} Y = \emptyset$ and so, since $Y$ is discrete and $Z\subseteq Y$,  $Z$ is closed. But
again, since $Z$ is discrete,
it means that $Z$ is finite and so almost all elements of $(y_n)$ belongs to $W$. As $W$ was arbitrary, $(y_n)$ converges to $x$.
 \end{proof}

\begin{prop}
\label{point-not-uni-noC}
If $\mathbb{A}$ is an infinite superatomic algebra, then there is a sequence $(\varphi_n)$ of homomoprhisms $\varphi_n\colon \mathbb{A} \to \mathbb{M}$ convergent pointwise but not uniformly. 
\end{prop}

\begin{proof}
	It follows from the proposition above and Proposition \ref{uniform-conv} as ultrafilers on \(\mathbb{A}\) correspond to homomorphisms into the trivial algebra (which is of course a subalgebra of $\mathbb{M}$).
\end{proof}

\begin{thm}\label{nontrivii}
	For every infinite Boolean algebra \(\mathbb{A}\)  there is a sequence of homomorphisms $\varphi_n\colon \mathbb{A} \to \mathbb{M}$ which converges nontrivially. 
\end{thm}

\begin{proof}
Apply Proposition \ref{point-not-uni-withC} and Proposition \ref{point-not-uni-noC}.
\end{proof}

\bibliographystyle{alpha}
\bibliography{bib-homo}

\end{document}